\documentclass[11pt,a4paper]{scrartcl}
\usepackage[utf8]{inputenc}
\usepackage[english]{babel}
\usepackage{helvet}
\usepackage{amsmath}
\usepackage{amsthm}
\usepackage{amsfonts}
\usepackage{amssymb}
\usepackage{mathtools}
\usepackage{mathrsfs}
\usepackage{dsfont}
\usepackage{hyperref}
\usepackage[noabbrev,capitalize]{cleveref}
\usepackage{tabularx}
\usepackage{xcolor}
\usepackage[paper=a4paper,left=25mm,right=25mm,top=25mm,bottom=25mm]{geometry}
\author{Helena Kremp}

\newtheorem{theorem}{Theorem}[section]
\newtheorem{definition}[theorem]{Definition}     
\newtheorem{proposition}[theorem]{Proposition}
\newtheorem{lemma}[theorem]{Lemma}	
\newtheorem{assumption}[theorem]{Assumption}
\crefname{assumption}{Assumption}{Assumptions}	

\newtheorem{corollary}[theorem]{Corollary}
\newtheorem{example}[theorem]{Example}

\newtheorem{remark}[theorem]{Remark}

\numberwithin{equation}{section}

\newcommand{\R}{\mathbb{R}} 
\newcommand{\N}{\mathbb{N}} 
\newcommand{\E}{\mathds{E}} 
\newcommand{\p}{\mathds{P}} 
\newcommand{\F}{\mathcal{F}} 

\renewcommand{\mathcal}[1]{\mathscr{#1}}
\renewcommand{\epsilon}{\varepsilon}

\newcommand{\assign}{:=}

\DeclarePairedDelimiter{\abs}{\lvert}{\rvert}
\DeclarePairedDelimiter{\norm}{\lVert}{\rVert}

\DeclarePairedDelimiter{\paren}{(}{)}

\begin{document}
\title{Weak error analysis for a nonlinear SPDE approximation of the Dean-Kawasaki equation}

\author{Ana Djurdjevac \footnote{Freie Universität Berlin, Arnimallee 6, 14195 Berlin, Germany; 	adjurdjevac@zedat.fu-berlin.de} \and Helena Kremp \footnote{Freie Universität Berlin, Arnimallee 7, 14195 Berlin, Germany; helena.kremp@fu-berlin.de} \and Nicolas Perkowski \footnote{Freie Universität Berlin, Arnimallee 7, 14195 Berlin, Germany; perkowski@math.fu-berlin.de}}

\maketitle
\begin{abstract}
\noindent We consider a nonlinear SPDE approximation of the Dean-Kawasaki equation for independent particles. Our approximation satisfies the physical constraints of the particle system, i.e. its solution is a probability measure for all times (preservation of positivity and mass conservation). Using a duality argument, we prove that the weak error between particle system and nonlinear SPDE is of the order $N^{-1-1/(d/2+1)}\log N$. Along the way we show well-posedness, a comparison principle, and an entropy estimate for a class of nonlinear regularized Dean-Kawasaki equations with It\^o noise.
\end{abstract}

\begin{section}{Introduction}
The Dean-Kawasaki (DK) equation, named after \cite{Dean1996} and \cite{Kawasaki}, is a stochastic partial differential equation (SPDE) for the evolution of the empirical measure of particles following Langevin dynamics with pairwise interaction. It is formally given by
\begin{align}\label{eq:DK-interaction}
\partial_{t}u=\frac{1}{2}\Delta u+\nabla\cdot(u (\nabla W\ast u))+\frac{1}{\sqrt{\alpha}}\nabla\cdot (\sqrt{u}\xi),
\end{align}
where $\xi$ is a vector-valued space-time white noise and $\alpha>0$. In \cite{Dean1996}, \eqref{eq:DK-interaction} is formally derived as a closed equation for the dynamics of the empirical measure of mean-field interacting diffusions: Let $(X^{i})_{i=1}^{N}$ solve the Langevin dynamics with interaction potential $W$, that is
\begin{align}\label{eq:Langevin}
dX^{i}_{t}=-\frac{1}{N}\sum_{j=1}^{N}\nabla W(X^{i}_{t}-X^{j}_{t})dt+dB^{i}_{t}
\end{align}
for $N$ independent Brownian motions $(B^{i})_{i=1}^{N}$. Dean \cite{Dean1996} argues that the empirical measure $\mu^N_t = \frac1N \sum_{i=1}^N \delta_{X^i_t}$ solves~\eqref{eq:DK-interaction} with $\alpha=N$. This is based on an ad-hoc replacement of the (stochastic integral against the) Brownian noise by a (stochastic integral against a) space-time white noise which formally has the same law.

\noindent Due to the singular gradient noise and because of the square root nonlinearity, the mathematical meaning of \eqref{eq:DK-interaction} is dubious, although there has been a lot of progress in recent years, see the literature review below. Nonetheless, the Dean-Kawasaki equation is very useful for applications and SPDE models of DK-type are very popular in physics (e.g.~\cite{Marconi2000,kreimer,Goddard, delfau2016,dejardin2018}), where they are also known as (stochastic) dynamical density functional theory, see the recent survey \cite{teVrugt2020}.

\noindent The motivation for this work is the successful use of the Dean-Kawasaki equation as a computational tool for simulations of fluid dynamics \cite{Donev2014} or social dynamics~\cite{DjCDjK, HDjCDjS}. If $N$ is large and if the particles live in a low-dimensional space, a discretization of \eqref{eq:DK-interaction} can be computationally much cheaper than simulating $N$ particles with their interactions. This is particularly interesting for moderate sizes of $N$, where the hydrodynamic limit ($N \to \infty$) would miss important stochastic fluctuations; for example, in social dynamics it is interesting to consider $N=O(10^4-10^5)$. However, because of the singular nature of the DK equation and because of its fragile solution theory (see the literature section below), it is far from obvious why a discretization of \eqref{eq:DK-interaction} should  give meaningful results for \eqref{eq:Langevin}.

\noindent To understand this on a conceptual level, we focus on the case without interaction ($W=0$) and we introduce a regularized version of the DK-equation, where we truncate the noise, we  mollify the square root nonlinearity and we consider on $\R_+ \times \mathbb T^d$ the equation
\begin{align}\label{eq:rDK}
d \tilde \mu^{N}_{t}=\frac{1}{2}\Delta \tilde \mu^{N}_{t}dt+\frac{1}{\sqrt{N}}\sum_{\abs{k}\leqslant M_{N}}\nabla (f_{\delta_{N}}(\tilde \mu_{t}^N)e_{k})\cdot dB_{t}^{k}
\end{align}
for the Fourier-basis $(e_k)_{k \in \mathbb Z^d}$, for a Lipschitz function $f_{\delta_N}$ which approximates the square root, and for $d$-dimensional complex Brownian motions $(B^k)_{k \in \mathbb Z^d}$. Under appropriate conditions we prove the well-posedness of this approximate DK-equation. The equation is not locally monotone in the sense of \cite{Liu2015} and the nonlinear It\^o gradient noise is quite tricky (much more so than Stratonovich noise), therefore we prove well-posedness through a suitable transformation of the equation in combination with a priori energy bounds. Furthermore, we prove a comparison principle which yields non-negativity of the solution for non-negative initial data and the conservation of the $L^{1}$-norm. This is important because $\tilde \mu^N$ should represent the empirical distribution $\mu^N$ of particles, which is a probability measure by the mass conservation of the particle system. Our main result concerns the quality of approximation of $\mu^N$ by $\tilde \mu^N$. We do not expect that $\tilde \mu^N$ and $\mu^N$ are pathwise close and we do not even couple the noises in their dynamics to each other. Instead, we focus on the weak approximation quality, and we show that for suitable test functions $F$
\begin{align}
	\abs{\E[F(\tilde{\mu}^{N}_{t})]-\E[F(\mu^{N}_{t})]}\lesssim N^{-1-\frac{1}{d/2+1}}(t+\log(N)).
\end{align}
Since the equation for $\tilde \mu^N$ should be amenable to discretization, this gives a partial justification for the use of DK-type equations in numerical simulations. Although our methods crucially rely on the absence of interactions because then we can compute $\E[F(\mu^{N}_{t})]$ explicitly by duality arguments (cf.~\cite{Konarovskyi2019}), we expect similar results to hold in the interacting case, and we expect that our general approach is quite powerful and can possibly be generalized to the interaction case.

\noindent The paper is structured as follows. Below we review some mathematical literature on the Dean-Kawasaki equation and we discuss why its solution theory is subtle and some recent breakthroughs. \cref{sec:dk-pre} introduces the notation and assumptions for our approximate Dean-Kawasaki equation, and we state the main results. In \cref{sec:dk-well,sec:dk-com} we prove the well-posedness and a comparison principle for regularized Dean-Kawasaki-type equations, which in particular applies to our approximation \eqref{eq:rDK}. \cref{sec:error-est} provides the weak error estimates.

\paragraph{Some mathematical literature on Dean-Kawasaki type equations}

\noindent Giving a pathwise meaning to the DK equation \eqref{eq:DK-interaction} is an open problem even in the case without interaction potential, $W=0$. The equation is highly singular because the noise is effectively a derivative of the space-time white noise and thus extremely irregular, and therefore the solution can only be a distribution and not a function and thus the square root $\sqrt{\mu^N_t}$ and the product $\sqrt{\mu^N_t}\xi$ have no meaning. Even in dimension $d=1$, the equation is scaling supercritical in the language of regularity structures (\cite{Hairer2014}) and paracontrolled distributions (\cite{Gubinelli2015Paracontrolled}), which are techniques to tackle subcritical singular SPDEs.

\noindent Entropy solutions of approximate DK-type equations with Stratonovich noise $\sum_{k}\nabla \cdot \sigma^{k}(x,u)\circ dB^{k}_{t}$ were studied in \cite{dareiotis2020nonlinear}. Using a pathwise approach via rough paths techniques, \cite{fehrman2019well} established well-posedness for porous media equations (replacing $\Delta u$ by $\Delta (\abs{u}^{m-1}u)$ with exponent $m\in (0,\infty)$), again for Stratonovich noise. Both of these results require very regular diffusion coefficients $\sigma$.
The well-posedness of the Dean-Kawasaki equation with truncated Stratonovich space-time white noise and with square-root diffusion coefficient was recently shown in \cite{fehrman2021well}, where the authors establish existence and uniqueness of stochastic kinetic solutions. The Stratonovich noise enables them to obtain a priori entropy-type estimates on the kinetic solutions (cf. \cite[Section 5.1]{fehrman2021well}), which yield the relative compactness of solutions with regularized square root nonlinearity. We refer to \cite{fehrman2021well} and the references therein for more details on the well-posedness of related equations.

\noindent Furthermore, the Dean-Kawasaki equation is related to scaling limits of interacting particle systems, see e.g. \cite{giacomin1998}. \cite{fg-ldp} show that the Stratonovich DK equation correctly predicts the large deviation rate function for non-equilibrium fluctuations of the zero range process.

\noindent Dean-Kasawaki type equations with regularized square root and truncated It\^o noise have been considered also in \cite{Bechtold}, where the existence of weak solutions is shown.

\noindent Regularized Dean-Kawasaki equations for underdamped kinetic particles were investigated in \cite{cornalba2019regularized, Cornalba2020-reg}. The authors impose a cutoff on the noise, which is formally justified by an approximation of the Dirac deltas in $\mu^N_t = \sum_{i=1}^N \delta_{X^i_t}$ with mollifying kernels. They establish the existence of mild solutions with high probability.

\noindent The mathematically correct interpretation of the ``full'' (not truncated) Dean-Kawasaki equation was found only recently~\cite{Konarovskyi2019, Konarovskyi2020}. The authors show that~\eqref{eq:DK-interaction} should be interpreted as a martingale problem, and that by similar duality arguments as for superprocesses~\cite{Etheridge2000} uniqueness holds for this martingale problem. In fact, the authors prove much more. Namely, at least for bounded interaction potentials, the Dean-Kawasaki equation has a unique martingale solution if and only if the parameter $N$ is a natural number and the initial condition is an atomic probability measure (given by the sum of $N$ Dirac measures). If we replace $\mu_{0}^{N}$ by an initial condition $\mu_{0}\in\mathcal{M}$ that is not of the form $\frac1N \sum_{i=1}^N \delta_{y^i}$, or if we replace $N$ by a constant $\alpha\in\R^{>0}\setminus\N$, then there exists no martingale solution. In particular, the well-posedness is very fragile with respect  to changes of the parameters, and as such, the equation is not suitable for a stable numerical approximation. 

\noindent From a numerical perspective, this instability in the parameter $N$ is worrisome. That is, slight changes in $N$ yield an ill-posed problem and possibly to large numerical errors. Nonetheless, numerical schemes 
for Dean-Kawasaki-type equations were considered in \cite{CornalbaShardlow, cornalba2021dean}. \cite{CornalbaShardlow} introduce a discontinuous Galerkin scheme for the regularized DK equation from \cite{cornalba2019regularized,Cornalba2020-reg}. The work most related to ours is Cornalba-Fischer \cite{cornalba2021dean}, who consider finite element and finite difference approximations of \eqref{eq:DK-interaction} without interaction ($W=0$) and who also prove weak error estimates. Their weak distance is parametrized by the Sobolev regularity of the test functions and the rate measured in their distance can be arbitrarily high, only limited by the numerical error and the error coming from the negative part of the approximation. However, the authors do not prove positivity for the approximations (hence the consideration of the negative part). Additionally they impose a strong assumption (\cite[Assumption FD4]{cornalba2021dean}) on the existence of lower and upper bounds for the solution of a discrete heat equation, which means that initially the particles must be fairly spread out. Our results are orthogonal in the sense that we consider an SPDE approximation and not a discrete model, and we do not have any restrictions on the initial distribution of particles. Our convergence rate is upper bounded by $N^{-5/3}$ in $d=1$ and it gets worse in higher dimensions, which is due to constraints on the choice of parameters coming from the solution theory of the SPDE. Our methods are very different (Laplace duality/Kolmogorov backward equation compared to a clever recursive scheme and choice of distance in \cite{cornalba2021dean}).
\end{section}

\begin{section}{Preliminaries and statement of the main results}\label{sec:dk-pre}
Consider the particle system of $N \in \N$ independent standard $d-$dimensional Brownian motions $(X^{i})_{i=1}^{N}$ projected onto the torus $\mathbb{T}^{d}=\R^{d}/\mathbb{Z}^{d}$ (with generator being the Laplacian with periodic boundary conditions) 
started at $X^{i}_{0}=x^{i}\in\mathbb{T}^{d}$. We are interested in the empirical measure
\begin{equation}\label{def:muN}
\mu^{N}_{t}:=\frac{1}{N}\sum_{i=1}^{N}\delta_{X^{i}_{t}}.
\end{equation}
\noindent Applying It\^{o}'s formula  to $\langle\mu^{N}_{t},\varphi\rangle=\frac{1}{N}\sum_{i=1}^{N}\varphi(X_{t}^{i})$, $\varphi\in C^{\infty}(\mathbb{T}^{d})$, we see that $\mu^N$ formally solves the Dean-Kawasaki equation without interaction and with atomic initial condition,
\begin{equation}\label{eq:DK}
\begin{aligned}
\partial_{t}\mu^N &=\frac{1}{2}\Delta\mu^N+\frac{1}{\sqrt{N}}\nabla\cdot (\sqrt{\mu^N}\xi)\\
\mu_{0}^{N} &=\frac{1}{N}\sum_{i=1}^{N}\delta_{x^{i}},
\end{aligned}
\end{equation}
where $\xi :=(\xi^{j})_{j=1,\dots,d}$ with independent space-time white noise processes $\xi^{j}$, see \cite{Dean1996}.  As in \cite[Definition 2.1]{Konarovskyi2019} we will interpret \eqref{eq:DK} as a martingale problem. The state space of the solution process is the space $\mathcal{M}$ of probability measures on $\mathbb{T}^{d}$, equipped with the topology of weak convergence. 
For $\mu\in\mathcal{M}$ and $\varphi\in C(\mathbb{T}^{d})$, we write $\mu(\varphi)=\langle\mu,\varphi\rangle=\int_{\mathbb{T}^{d}}\varphi d\mu$.
\begin{definition}\label{def:DK}
We call a stochastic process $(\mu_t^{N})_{t\geqslant 0}$ on a complete filtered probability space $(\Omega, \F, (\F_{t})_{t\geqslant 0},\p)$ with values in $C(\R_+,\mathcal{M})$
a solution to \eqref{eq:DK} if for any test function $\varphi \in C^{\infty}(\mathbb{T}^d)$, the process
\begin{align*}
t\mapsto\langle\mu_{t}^{N},\varphi\rangle-\langle\mu_{0}^{N},\varphi\rangle-\int_{0}^{t}\langle\mu_{s}^{N},\frac{1}{2}\Delta\varphi\rangle ds
\end{align*}
is an $(\F_{t})$--martingale with quadratic variation 
\begin{equation}
\frac{1}{N}\int_{0}^{\cdot}\langle\mu_{s}^{N},\abs{\nabla\varphi}^{2}\rangle ds.
\end{equation}
\end{definition}

\noindent By \cite[Theorem 2.2]{Konarovskyi2019} this martingale problem has a unique (in law) solution given by (\ref{def:muN}). And if in \eqref{eq:DK} we replace $\mu_{0}^{N}$ by an initial condition $\mu_{0}\in\mathcal{M}$ that is not of the form $\frac1N \sum_{i=1}^N \delta_{y^i}$, or if we replace $N$ by a constant $\alpha\in\R^{>0}\setminus\N$, then there exists no martingale solution to the equation. In particular, the well-posedness is very fragile with respect  to changes of the parameters. Our goal is to approximate the equation  \eqref{eq:DK} in a controlled manner with an equation that has good stability properties and that preserves the physical constraints of $\mu^N$, i.e. positivity and mass conservation.
	
\noindent We aim for a bound of the form
\begin{align*}
\abs{\E[F(\tilde{\mu}^{N}_{t})]-\E[F(\mu^{N}_{t})]}\lesssim_{F,t} N^{-\alpha_{d}},
\end{align*} 
for $t>0$ and for suitable nonlinear test functions $F : \mathcal M \to \R$. As usual, $a \lesssim b$ means that there exists a constant $C>0$ (not depending on the relevant parameters; above the parameters are $\delta,x$), such that $a\leqslant C b$. If we want to indicate the dependence of the constant $C(\kappa)$ on a parameter $\kappa$, we write $a\lesssim_{\kappa}b$. We will see that for $F(\mu):=\exp(\langle\mu,\varphi\rangle)$, $\varphi\in C^{\infty}(\mathbb{T}^{d})$ the expectation $\E[F(\mu^{N}_{t})]$ can be computed in closed form and therefore we consider such $F$. 
\noindent Due to the factor $1/N$ in front on the quadratic variation, a direct computation shows that the hydrodynamic/mean-field limit, that is, the solution $\rho$ of the heat equation
\begin{align*}
\partial_{t}\rho=\frac{1}{2}\Delta\rho,\qquad \rho_0 = \mu^N_0,
\end{align*} achieves a rate of convergence $\alpha_{d}=1$. However, $\rho$ is deterministic and it does not capture random fluctuations in the particle system $\mu^N$. For the Gaussian approximation $\hat{\rho}$ of the fluctuations of the particle system around the hydrodynamic limit, that is,
\begin{align*}
\partial_{t}\hat{\rho}=\frac{1}{2}\Delta\hat{\rho}+\nabla\cdot (\sqrt{\rho}\xi),\qquad \hat \rho_0 = 0,
\end{align*}
with the same vector-valued space-time white noise $\xi$ as above, we can prove that $\rho + \frac{1}{\sqrt N} \hat \rho$ achieves a weak error of $\alpha_{d}=3/2$. But $\rho + \frac{1}{\sqrt N} \hat \rho$ is a distribution of regularity $H^{-\frac{d}{2}-\epsilon}$, it is not a measure, let alone positive, and it does not conserve the mass of the initial condition (in fact $\norm{\rho_t + \frac{1}{\sqrt N} \hat \rho_t}_{L^1} = \infty$ for all $t > 0$). 

\noindent Therefore, we consider a nonlinear approximation of \eqref{eq:DK}: We replace the non-Lipschitz square root function with a Lipschitz approximation that will depend on a parameter $\delta$ and we replace the noise by its ultra-violet cutoff at frequencies of order $M$. The parameters $\delta,M$  will influence the order of the approximation and they will be chosen subsequently in the error estimate, depending on $N$. Let for now $\delta=\delta_{N}>0$ and $M= M_N \in \N$. We then define the Lipschitz function $f = f_{\delta}$ as follows
\begin{equation}\label{def:f}
f(x) = 
\begin{cases} 
      \frac{1}{\sqrt{\delta}}x & \abs{x}\leqslant \delta/2, \\
      \text{smooth} &   \delta/2\leqslant\abs{x}\leqslant \delta, \\
   \operatorname{sign}(x)\sqrt{\abs{x}} &\abs{x}\geqslant \delta.\\
\end{cases}
\end{equation}
The smooth interpolation should be such that $f\in C^{1}(\mathbb{R})$ satisfies  
\begin{align}\label{eq:f-prime-bounds}
\norm{f'}_{L^{\infty}}\lesssim\frac{1}{\sqrt{\delta}}, \quad 
\abs{f'(x)}\lesssim\frac{1}{\sqrt{x}},\text{ for all } x>0
\end{align}
and 
\begin{align}\label{eq:f-bounds}
\abs{f(x)}\lesssim\sqrt{\abs{x}}, \quad 
\abs{f(x)^{2}-x}\lesssim\delta, \text{ for all } x\geqslant 0.
\end{align} 
Any $C^{1}$ approximation of the square root satisfying those bounds works for our analysis and a particular example of such a function is given in the following example.
\begin{example}
Consider, for example,
\begin{equation}
f(x) = 
\begin{cases} 
      \frac{1}{\sqrt{\delta}}x & \abs{x}\leqslant \delta/2, \\
      - \frac{2 \sqrt{\delta}}{\delta^3} x^3 +  \operatorname{sign}(x)\frac{4}{\delta \sqrt{\delta}} x^2 -  \frac{3}{2 \sqrt{\delta}} x  + \operatorname{sign}(x)\frac{\sqrt{\delta}}{2} &   \delta/2\leqslant\abs{x}\leqslant \delta, \\
   \operatorname{sign}(x)\sqrt{\abs{x}} &\abs{x}\geqslant \delta.\\
\end{cases}
\end{equation}
It is not hard to see that $f\in C^{1}$ and that $f$ satisfies the bounds \eqref{eq:f-prime-bounds} and \eqref{eq:f-bounds}. 
\end{example}

\noindent We denote by $\tilde{\mu}^{N}$ the solution of the approximated equation 
\begin{equation}\label{def:app_mu}
\begin{aligned}
d\tilde{\mu}^{N}_{t} &=\frac{1}{2}\Delta\tilde{\mu}^{N}_{t}dt+\frac{1}{\sqrt{N}}\nabla\cdot (f(\tilde{\mu}^{N}_{t})dW^{N}_{t}), \\
\tilde{\mu}_{0}^{N} &=\rho^{N}\ast\mu_{0}^{N}\in L^{2}(\mathbb{T}^d),
\end{aligned}
\end{equation} 
where $(\rho^{N})_N$ is an appropriate approximation of the identity, that we will choose later. 
Moreover, the truncated noise $(W_t^N)_t$ is given by
\begin{equation}\label{app_noise}
W^{N}_{t}(x):=\sum_{\abs{k}\leqslant M_{N}}e_{k}(x)B_{t}^{k}:=\sum_{\abs{k}\leqslant M_{N}}\exp(2\pi i k\cdot x)B_{t}^{k}
\end{equation}
for $x\in\mathbb{T}^{d}$ and independent $d$-dimensional complex-valued Brownian motions $(B^{k})_{k\in\mathbb{Z}^{d}}$  (that is, $B^{k}=B^{k,1}+iB^{k,2}$ for independent $\R^{d}$-valued Brownian motions $B^{k,1},B^{k,2}$) with constraint $\overline{B^{k}}=B^{-k}$,
and a truncation parameter $M_{N}\in\N$, that will be chosen depending on $N$ for the error estimate.\\ 
In the next section we prove the strong well-posedness of slightly more general equations than \eqref{def:app_mu}, as well as the non-negativity of the strong solution and the mass conservation property. But first let us formulate a summary of our main results:

\begin{theorem}[Summary of the main results]
	Let $\tilde \mu^N_0 \in L^2$ be positive and let $M_N$ and $\delta_N $ be such that $\frac{C_N(2M_{N}+1)^{d}}{N\delta_N} < 1$, where $C_N>0$ is such that $\|f'\|_\infty^2 \leqslant \frac{C_N}{\delta_N}$. Then there exists a unique solution $\tilde \mu^N$ to \eqref{def:app_mu}, which is positive and which satisfies $\norm{\tilde \mu^N_t}_{L^1(\mathbb T^d)} = \norm{\tilde \mu^N_0}_{L^1(\mathbb T^d}$ for all $t \geqslant 0$, as well as the entropy bound for $\lambda:=\frac14(1-\frac{C_N(2M_{N}+1)^{d}}{N\delta_N})$:
	\begin{equation}
		\sup_{t\in[0,T]}\E\bigg[\int\tilde{\mu}^{N}_{t}\log(\tilde{\mu}^{N}_{t})\bigg]+\lambda\int_{0}^{T}\E\bigg[\int\frac{\abs{ \nabla\tilde{\mu}^{N}_{t}}^{2}}{\tilde{\mu}^{N}_{t}}\bigg]dt \lesssim \int\tilde{\mu}^{N}_{0}\log(\tilde{\mu}^{N}_{0}) + \frac{T M_{N}^{d+2}}{N},
	\end{equation}
	If $M_N = \delta_N^{-1/2}$ and $\delta_N \simeq N^{-\frac{1}{d/2+1}}$ with $\sup_N \frac{C_N(2M_{N}+1)^{d}}{N\delta_N} < 1$, and if $\mu^N$ is the martingale solution of \eqref{eq:DK} and if $\tilde \mu^N_0$ is an approximation of $\mu^N_0$ as in Lemma~\ref{lem:initial-condition-entropy}, then for any $t>0$, $\varphi\in C^{\infty}(\mathbb{T}^{d})$ and $F(\mu):=\exp(\langle\mu,\varphi\rangle)$ for $\mu\in\mathcal{M}$, the following weak error bound holds:
	\begin{equation}
		\abs{\E[F(\tilde{\mu}^{N}_{t})]-\E[F(\mu^{N}_{t})]}\lesssim_{\varphi} N^{-1-\frac{1}{d/2+1}}(t+\log(N)).
	\end{equation}
\end{theorem}

\begin{proof}
	All this is shown in Section~\ref{sec:error-est}. See \cref{prop:app_mu_wellposed} for the well-posedness, see  \cref{prop:entropy-est} for the entropy estimate, and see \cref{thm:error-est} for the weak error estimate.
\end{proof}

\noindent In summary, the nonlinear approximation $\tilde{\mu}^{N}$ achieves a weak rate $\alpha_{d}>1$, which is always better than the error of the deterministic approximation $\rho$. In $d=1$ the nonlinear SPDE has a smaller weak error than the Gaussian approximation $\rho + \frac{1}{\sqrt N} \hat \rho$ (rate $N^{-5/3}\log N$ compared to $N^{-3/2}$), and in $d=2$ the rates are nearly the same ($N^{-3/2} \log N$ respectively $N^{-3/2}$). In higher dimensions we get a worse error bound for $\tilde \mu^N$ than for the Gaussian approximation, which achieves $N^{-3/2}$ independently of the dimension. But on the positive side $\tilde \mu^N$ is a probability density, while $\rho + \frac{1}{\sqrt N} \hat \rho$ is not positive and only a Schwartz distribution and not even a signed measure. The main obstruction towards reaching a better rate for $\tilde \mu^N$ is that the solvability conditions of Section~\ref{sec:dk-well} impose constraints on $\delta_N$ and $M_N$.

\end{section}

\begin{section}{Well-posedness and comparison for regularized DK-type SPDEs}
\begin{subsection}{Well-posedness}\label{sec:dk-well}
In this section, we prove strong well-posedness for SPDEs of the type
\begin{equation}\label{eq:rSPDE}
\begin{aligned}
du_{t} &=\frac{1}{2}\Delta u_{t}dt+\nabla\cdot (b(u_{t})dW_{t}),\\
u_{0} &\in L^{2}.
\end{aligned}
\end{equation} 
Similar stochastic conservation laws with Stratonovich noise have been intensely studied in the past years, see for example \cite{dareiotis2020nonlinear, fehrman2019well, fehrman2021well}, among many others. However, due to the gradient noise the solution theory for It\^o noise is more subtle and we need to make an additional smallness assumption to even get the existence of solutions; see the discussion before \cref{ass:const} below. Also, the coefficients of \eqref{eq:rSPDE} are not locally monotone in the sense of \cite[Theorem 5.1.3]{Liu2015}, and therefore the solution theory does not follow from standard theory.

\noindent There is the related work \cite{Bechtold} by Bechtold on equations resembling \eqref{eq:rSPDE}, and we are in the case of ``critical unboundedness'' from \cite[Section 4]{Bechtold}.  However, \cite{Bechtold} only proves the existence of a probabilistically weak solution with paths in $C([0,T],H^{-\epsilon})\cap L^{2}([0,T],H^{1})$ for any $\epsilon>0$. By a pathwise uniqueness argument using the a priori energy bound that we derive below, it should be possible to show the strong existence and uniqueness of a solution with paths in  $C([0,T],H^{-\epsilon})\cap L^{2}([0,T],H^{1})$ for any $\epsilon>0$. 
Instead, we  directly show the stronger statement of strong existence and uniqueness of a solution with paths in $C([0,T],L^{2})\cap L^{2}([0,T],H^{1})$.

\noindent We make the following assumptions.
\begin{assumption}[Assumption on the noise]\label{ass:noise}
Let $(\phi_{k})_{k}\subset C^{1}(\mathbb{T}^{d},\mathbb{C})$ with $\overline{\phi_{k}}=\phi_{-k}$ 
and consider independent $d$-dimensional complex-valued Brownian motions $(B^{k})_{k\in\mathbb{Z}^{d}}$ with $\overline{B^{k}}=B^{-k}$. We assume that the noise $(W_t)_t$ is given by
\begin{align*}
W_{t}(x):=\sum_{k\in\mathbb{Z}^{d}}\phi_{k}(x)B^{k}_{t}
\end{align*} 
such that 
\begin{align}\label{const_bounds}
C^{W}_{1}:=\sum_{k\in\mathbb{Z}^{d}}\norm{\phi_{k}}_{\infty}^{2}<\infty\quad\text{ and }\quad C^{W}_{2}:=\sum_{k\in\mathbb{Z}^{d}}\norm{\nabla\phi_{k}}_{\infty}^{2}<\infty,
\end{align} with $\norm{\phi_{k}}_{\infty}^2:=\sup_{x\in\mathbb{T}^{d}}\abs{\phi_{k}(x)}^2$ 
and $\norm{\nabla\phi_{k}}_{\infty}^{2}:=\sup_{x\in\mathbb{T}^{d}}\sum_{i=1}^{d}\abs{\partial_{i}\phi_{k}(x)}^2$. For some results we only require $C^W_1<\infty$, and we call this the \emph{relaxed \cref{ass:noise}}.
\end{assumption}

\begin{remark}
An example for a noise expansion satisfying the assumption is the Fourier expansion with cut-off $M_{N}\in\N$ from \eqref{app_noise}. The summability assumptions are trivially satisfied due to the finite cut-off. 
Similar to \cite[Remark 2.3]{fehrman2021well}, we can also consider a noise \begin{align*}
W_t(x)=\sum_{k\in\mathbb{Z}^{d}}a_{k}\exp(2\pi i k\cdot x)B^{k}_t,
\end{align*} for a real sequence $(a_{k})$ with $\sum_{k\in\mathbb{Z}^{d}}\abs{k}^{2}a_{k}^{2}<\infty$, which also satisfies our assumption.
\end{remark}

\begin{assumption}[Assumptions on the diffusion coefficient]\label{ass:diff}
We assume that the diffusion coefficient  $b\in C^{1}(\R,\R)$ has a bounded derivative with bound $L>0$,
\begin{align}\label{eq:Lip}
\norm{b'}_{\infty}\leqslant L.
\end{align} 
In particular, $b$ is of linear growth: There exist $C^b_{1}, C^b_2$ such that 
\begin{align}\label{eq:lin-g}
\abs{b(x)}^{2}\leqslant  C^b_1 \abs{x}^{2} + C^b_2, \qquad  x\in\R.
\end{align} 
\end{assumption}
\noindent From the assumption \ref{ass:diff}, it follows that $b$ is Lipschitz continuous with bound $L$, that is $|b(x) - b(y)| \leq L |x-y|$ for all $x,y \in \R$. 

\noindent If we would consider Stratonovich noise, these two assumptions would be sufficient. For It\^o noise we need an additional smallness condition for the gradient noise, because otherwise it could break the parabolic nature of the equation; see \cite[Section~2.4.2]{Pardoux2021}.

\begin{assumption}[Stochastic parabolicity]\label{ass:const}
We assume that the parameters from the previous assumptions satisfy $C_{1}^W\max(L^2,C_1^b)<1$.
\end{assumption}

\noindent The well-posedness theory of this section will apply to the approximating Dean-Kawasaki equation \eqref{def:app_mu}, where $b=\frac{1}{\sqrt{N}}f$ with $f$ from \eqref{def:f} and $L=C^b_1=1/\sqrt{N\delta}$. Furthermore,  the noise expansion is given with respect to the Fourier basis $(\phi_{k})_{k}$ with cut-off $M_{N}\in\N$. In that case, we have that $C_{1}^W\leqslant (2M_{N})^{d}$ and $C_{2}^W\leqslant(2M_{N})^{d}(2\pi M_{N})^{2}$. Hence, \eqref{const_bounds} is fulfilled. 

\noindent Note that due to the gradient noise term the local monotonicity condition is violated for  SPDEs of the form \eqref{eq:rSPDE} and the variational approach of \cite{Liu2015} cannot be applied directly. Instead, we transform the equation (by applying $(\Delta-1)^{1/2}$) into an equation for which the variational theory can be applied and we deduce well-posedness of the original equation by using a priori energy bounds.

\noindent The setting for the variational theory is defined as follows. Let $V=H^{1}(\mathbb{T}^{d})$ with $V^{\ast}=H^{-1}(\mathbb{T}^{d})$ and $H=L^{2}(\mathbb{T}^{d})$, such that we have the Gelfand triple $V\subset H\subset V^{\ast}$. We consider the Laplacian with periodic boundary conditions, that is, $\Delta: V=H^{1}(\mathbb{T}^{d})\to (H^{1}(\mathbb{T}^{d}))^{\ast}=V^{\ast}$ with $\Delta u (v):=\langle \Delta u, v\rangle_{V^{\ast},V}$, $u\in V^{\ast}, v\in V$. For the duality pairing, we have that $\langle \Delta u, v\rangle_{V^{\ast},V}=-\langle \nabla u, \nabla v\rangle_{H}=-\int\nabla u\cdot\nabla v$ when $u,v\in V$. Then, we define a solution to the equation (\ref{eq:rSPDE}) as follows.

\begin{definition}\label{def:app-sol}
A stochastic process $(u_{t})_{t\geqslant 0}$ with paths in $L^{2}([0,T],H^{1}(\mathbb{T}^d))\cap C([0,T],L^{2}(\mathbb{T}^d))$ is a (probabilistically strong and analytically weak) solution to the equation \eqref{eq:rSPDE} for the initial condition $u_{0}\in L^{2}(\mathbb{T}^d)$, if for all 
$\varphi\in H^{1}(\mathbb{T}^d)$,
\begin{align}\label{eq:app-sol}
\langle u_{t},\varphi\rangle&=\langle u_{0},\varphi\rangle-\int_{0}^{t}\frac{1}{2}\langle\nabla u_{s},\nabla\varphi\rangle ds+\sum_{k\in\mathbb{Z}^{d}}\int_{0}^{t}\langle\nabla(b(u_{t})\phi_{k}),\varphi\rangle\cdot dB^{k}_{s}
\\&=\langle u_{0},\varphi\rangle-\sum_{i=1}^{d}\int_{0}^{t}\frac{1}{2}\langle\partial_{i} u_{s},\partial_{i}\varphi\rangle ds+\sum_{i=1}^{d}\sum_{k\in\mathbb{Z}^{d}}\int_{0}^{t}\langle\partial_{i}(b(u_{t})\phi_{k}),\varphi\rangle dB^{k,i}_{s}.\nonumber
\end{align} 
\end{definition}

\noindent First, we study well-posedness of a transformed equation with the variational approach:

\begin{lemma}\label{axillary}
Let Assumptions  \ref{ass:diff} and \ref{ass:const} hold, as well as the relaxed Assumption~\ref{ass:noise}, and consider an initial condition $v_{0}\in L^{2}(\mathbb{T}^d)$. Then, there exists a unique probabilistically strong solution $v\in L^{2}([0,T],H^{1}(\mathbb{T}^d))\cap C([0,T],L^{2}(\mathbb{T}^d))$ to
\begin{align}\label{eq:v-VT}
d\langle v_{t},\varphi\rangle=-\frac{1}{2}\langle\nabla v_{t},\nabla\varphi\rangle dt+\sum_{k\in\mathbb{Z}^{d}}\langle G_{k}(v_{t}),\varphi\rangle \cdot dB^{k}_{t}
\end{align} 
for all $\varphi\in H^{1}$, where $G_{k}(v):=(1-\Delta)^{-1/2}\nabla (b((1-\Delta)^{1/2}v)\phi_{k})$.
\end{lemma}

\begin{proof}
We will  check the conditions of \cite[Section 4]{Liu2015} to apply the variational theory.
Let us define $G(v) w \assign \sum_{k\in\mathbb{Z}^{d}} (1 - \Delta)^{- 1 / 2} \nabla (b ((1 -\Delta)^{1 / 2} v) \phi_k) w_k$ for $w\in l^{2}(\mathbb{Z}^{d})$ and $v\in H^{1}$.
Using \eqref{eq:lin-g}, we obtain the  following coercivity bound (that is, (H3) from \cite[Section 4]{Liu2015}): 
\begin{align*}
  \left\langle  \Delta v, v \right\rangle_{V^{\ast},V} + \| G (v)
  \|^2_{L_2 (l^{2}(\mathbb{Z}^{d}), H)} & = - \| \nabla v
  \|^2_{L^2} +  \sum_{k} \| (1 - \Delta)^{- 1 / 2} \nabla
  (b ((1 - \Delta)^{1 / 2} v) \phi_k) \|_{L^2}^2\\
  & \leqslant -  \| \nabla  v \|^2_{L^2} + \sum_{k} \|
  b ((1 - \Delta)^{1 / 2} v) \phi_k \|_{L^2}^2\\
  & \leqslant -  \|  \nabla v \|^2_{L^2} + \sum_{k} \|
  \phi_k \|_{\infty}^2 \norm{ b ((1 - \Delta)^{1 / 2} v)}_{L^{2}}^{2}\\
  & \leqslant - \|  \nabla v \|^2_{L^2} +  C_{1}^W( C_1^b \norm{ (1 - \Delta)^{1 / 2}
  v}_{L^{2}}^{2} + C^b_2) \\
  & = (C_{1}^W C^b_1-1)  \|  \nabla v \|^2_{L^2} + C_1^W C_1^b\|v\|_{L^2}^2 + C_{1}^W C^b_2,
\end{align*}
where we used the Assumption \ref{const_bounds}. Here, the first inequality follows from the Plancherel theorem for the Fourier transform $\F_{\mathbb{T}^{d}}$ on the torus ($\F_{\mathbb{T}^{d}}f(k):=\int_{\mathbb{T}^{d}}e^{- 2\pi i k\cdot x}f(x)dx$), such that for $f\in L^{2}(\mathbb{T}^{d})$,
\begin{align*}
\norm{(1 - \Delta)^{- 1 / 2} \nabla f}_{L^{2}(\mathbb{T}^{d})}^2 = \sum_{k\in\mathbb{Z}^{d}}(1+\abs{2\pi k}^2)^{-1}\abs{2\pi k}^2\abs{\F_{\mathbb{T}^{d}}(f)(k)}^2\!\leqslant  \!\sum_{k\in\mathbb{Z}^{d}}\abs{\F_{\mathbb{T}^{d}}(f)(k)}^2=\norm{f}_{L^{2}(\mathbb{T}^{d})}^2,
\end{align*}
and similarly we get $\norm{ (1 - \Delta)^{1 / 2}v}_{L^{2}}^{2} = \norm{ \nabla v}_{L^{2}}^{2}+\norm{v}_{L^2}^2$. The coercivity then follows by the Assumption \ref{ass:const} on the parameters, since $C_1^W C_1^b < 1$. 
The weak monotonicity condition (that is, condition (H2) from \cite[Section 4]{Liu2015}) follows from the analogue estimate,  using the global Lipschitz bound $L$ due to \eqref{eq:Lip}:
\begin{align*}
  & \left\langle  \Delta (v^1 - v^2), v^1 - v^2 \right\rangle_{V^{\ast},V} +
   \| G (v^1) - G (v^2) \|^2_{L^2 (l^{2}(\mathbb{Z}^{d}), H)}\\
  & = -  \|  \nabla (v^1 - v^2) \|_{L^2}^2 +  \sum_{k}
  \| (1 - \Delta)^{- 1 / 2} \nabla ((b ((1 - \Delta)^{1 / 2} v^1_t) -
  b ((1 - \Delta)^{1 / 2} v^2_t)) \phi_k) \|_{L^2}^2\\
  & \leqslant -  \|  \nabla (v^1 - v^2) \|_{L^2}^2 + 
  \sum_{k} \| (b ((1 - \Delta)^{1 / 2} v^1_t) - b ((1 - \Delta)^{1 / 2} v^2_t))
  \phi_k \|_{L^2}^2\\
  & \leqslant -  \|  \nabla (v^1 - v^2) \|_{L^2}^2 + 
  \sum_{k} \| \phi_k \|_{\infty}^2 L^{2} \| (1 -
  \Delta)^{1 / 2} (v^1_t - v^2_t) \|_{L^2}^2\\
  & = - \|  \nabla (v^1 - v^2) \|_{L^2}^2 + 
  C_{1}^W L^{2} (\|
  v^1_t - v^2_t \|_{L^2}^2 + \|  \nabla (v^1_t - v^2_t) \|_{L^2}^2)
  \\&\leqslant  C_{1}^WL^{2} \|
  v^1_t - v^2_t \|_{L^2}^2,
\end{align*}
using $C_1^W L^2 < 1$ in the last step. Thus, we obtain the weak monotonicity.
The hemicontinuity  from \cite[Section 4, H1]{Liu2015} follows by linearity of the Laplacian. The boundedness condition  from \cite[Section 4, H4]{Liu2015} is trivially satisfied due to continuity, that is, $\norm{\Delta u}_{V^{\ast}}\leqslant\norm{v}_{V}$.
\end{proof}

\begin{remark}\label{equivalence}
If $u$ is a solution to \eqref{eq:app-sol}, then $v:=(1-\Delta)^{-1/2}u$ is a solution to \eqref{eq:v-VT}.  
As \eqref{eq:v-VT} has a unique strong solution, it follows that the solution to \eqref{eq:app-sol} is pathwise unique.

\noindent Moreover, if in \cref{def:app-sol} we formally integrate $\langle \nabla(b(u_t)\phi_k),\varphi\rangle = \langle b(u_t)\phi_k,\nabla\varphi\rangle$ by parts, the definition makes sense even for $u$ with paths in $L^{2}([0,T],L^2(\mathbb{T}^d))\cap C([0,T],H^{-1}(\mathbb{T}^d))$. With this definition the equation is  equivalent to Equation \eqref{eq:v-VT} for $v:=(1-\Delta)^{-1/2}u$. Therefore, under Assumptions  \ref{ass:diff} and \ref{ass:const} and the relaxed Assumption~\ref{ass:noise}, for any $u_0 \in H^{-1}(\mathbb T^d)$ there exists a pathwise unique strong solution $u$ with paths in $L^{2}([0,T],L^2(\mathbb{T}^d))\cap C([0,T],H^{-1}(\mathbb{T}^d))$ to \eqref{eq:app-sol}. We only need the full \cref{ass:noise} to obtain better regularity for $u$. 

\noindent Note that the relaxed \cref{ass:noise} corresponds to a subcriticality condition in the sense of regularity structures~\cite{Hairer2014}: Indeed, this is precisely what we need from the noise so that the solution to the linearized equation $dZ_t = \frac12 \Delta Z_t dt + \nabla \cdot dW_t$ is a function in the space variable and not a distribution. However, under such weak assumptions $Z_t$ will not be in $H^1$ and therefore also $u_t$ should not be in $H^1$. So, to obtain better regularity for $u_t$, we make stronger assumptions on the noise. Besides the subcriticality condition we also need the smallness condition \cref{ass:const} which is due to the It\^o noise. We expect that for Stratonovich noise the relaxed \cref{ass:noise} together with \cref{ass:diff} is sufficient and that with our transform we can solve the Stratonovich version of \eqref{eq:app-sol} in the entire subcritical regime.
\end{remark}

\noindent To prove existence of a solution to \eqref{eq:app-sol} in the sense of \cref{def:app-sol}, we proceed as follows: 
Given the solution $v$ of \eqref{eq:v-VT}, we can define $u:=(1-\Delta)^{1/2}v$. Then 
$u\in C([0,T],H^{-1}(\mathbb{T}^d))\cap L^{2}([0,T],L^{2}(\mathbb{T}^{d}))$ almost surely and, by the equation for $v$, $u$ is a solution to the (``very weak'') equation
\begin{align}\label{eq:H(-1)-form}
\langle u_{t},\varphi\rangle=\langle u_{0},\varphi\rangle+\int_{0}^{t}\frac{1}{2}\langle u_{s},\Delta\varphi\rangle ds-\sum_{k}\int_{0}^{t}\langle b(u_{t})\phi_{k}, \nabla\varphi\rangle\cdot dB^{k}_{s},
\end{align} 
for all $\varphi\in C^{\infty}(\mathbb{T}^{d})$. 
The lemma below proves an energy estimate which yields tightness of the sequence of Galerkin projected solutions $(u^{R})_R$. That is, for the orthogonal projection $\Pi_{R}:H^{1}\to V_{R}$ with $V_{R}:=\text{span}(e_{r}\mid r\in\mathbb{Z}^{d},\abs{r}\leqslant R)$ for the Fourier basis 
$(e_{r})_{r\in\mathbb{Z}^{d}}$, we let
\begin{align}\label{eq:Galerkin}
u^{R}:=(1-\Delta)^{1/2}\Pi_{R}v^{R}=\Pi_{R}(1-\Delta)^{1/2}v^{R}.
\end{align} 
Here, $v^{R}$ solves \eqref{eq:v-VT} with $G_{k}$ replaced by $G_{k}^{R}$ with 
\begin{align*}
G_{k}^{R}(v):=\Pi_{R}(1-\Delta)^{-1/2}\nabla(b(\Pi_{R}(1-\Delta)^{1/2}v)\phi_{k}),\quad v\in H^{1}.
\end{align*}
By the uniqueness of the solution $v$ and the construction of the solution in the variational theory (cf. \cite[Theorem 5.1.3]{Liu2015}), it follows that $v^{R}\to v$ in $L^{2}(\Omega\times[0,T],H^{1})$.

\begin{lemma}\label{lem:energy-est}
Let Assumptions \ref{ass:noise}, \ref{ass:diff} and \ref{ass:const} hold. Let  $u_{0}\in L^{2}(\mathbb{T}^d)$  and let $v_{0}:=(1-\Delta)^{-1/2}u_{0}$. Let $v^{R}$ be the solution of the equation \eqref{eq:v-VT} with $G_{k}$ replaced by $G_{k}^{R}$, let $v^{R}_{0}=\Pi_{R}v_{0}$, and let $u^{R}$ be defined as in \eqref{eq:Galerkin}.  
Then, the following energy bound holds true:
\begin{align}\label{eq:b1}
\E[\norm{u^{R}_{t}}_{L^{2}}^{2}]+\lambda\int_{0}^{t}\E[\norm{u^{R}_{s}}_{\dot{H}^{1}}^{2}]ds\leqslant 2C_{2}C^{2}t\norm{u_{0}}_{L^{2}}^{2}\exp(t\tilde{\lambda}),
\end{align} where 
$\lambda:=1-C_{1}^WL^{2}>0$ and $\tilde{\lambda}:=C_{2}^W C^{b}_1$.
\end{lemma}
\begin{proof}
Let $u^{R}=(1-\Delta)^{1/2}\Pi_{R}v^{R}=\Pi_{R}(1-\Delta)^{1/2}v^{R}$ for the projection $\Pi_{R}$ defined as above. As $e_{r}\in C^{\infty}(\mathbb{T}^{d})$, it follows that $\Pi_{R}v^{R}\in C([0,T],C^{\infty})$ almost surely and thus in particular $u^{R}\in C([0,T],C^{\infty})$ a.s. Furthermore, by the equation that $v^{R}$ solves, $u^{R}$ solves 
\begin{align}
\langle u_{t}^{R},\varphi\rangle=\langle u_{0}^{R},\varphi\rangle+\int_{0}^{t}\frac{1}{2}\langle \Delta u_{s}^{R},\varphi\rangle ds+\sum_{k}\int_{0}^{t}\langle \Pi_{R}\nabla (b(u_{t}^{R})\phi_{k}),\varphi\rangle\cdot dB^{k}_{s},
\end{align} 
for all $\varphi\in C^{\infty}(\mathbb{T}^{d})$, and since $u^{R}\in C([0,T],C^{\infty})$ it follows that
\begin{align*}
u^{R}_{t}(x)=u^{R}_{0}(x)+\int_{0}^{t}\frac{1}{2}\Delta u^{R}_{s}(x)ds+\sum_{k}\int_{0}^{t} \Pi_{R}\nabla (b( u^{R}_{s})\phi_{k}) (x) \cdot dB^{k}_{s}.
\end{align*}
By applying Itô's formula to $(u^{R}_{t}(x))^{2}$, we then obtain 
\begin{align*}
d\norm{u^{R}_{t}}_{L^{2}}^{2}=\int d(u^{R}_{t} (x))^{2} dx 
&=2\int u^{R}_{t}(x)\paren[\bigg]{\frac{1}{2}\Delta u^{R}_{t}(x)dt+\sum_{k} \Pi_{R}\nabla (b(u^{R}_{t}(x))\phi_{k} (x))\cdot dB^{k}_{t}}dx\\&
\quad +\sum_{k}\int \abs{ \Pi_{R}\nabla (b(u^{R}_{t}(x))\phi_{k}(x))}^{2}dxdt.
\end{align*}
Taking the expectation, the martingale vanishes and using $\norm{\Pi_{R}v}_{L^{2}}\leqslant\norm{v}_{L^{2}}$ and \eqref{eq:Lip}, we obtain
\begin{align}\label{eq:b-est}
\E[\norm{u^{R}_{t}}_{L^{2}}^{2}]&\leqslant\norm{u^{R}_{0}}_{L^{2}}^{2}-\int_{0}^{t}\E\bigg[\int_{\mathbb{T}^{d}}\nabla u^{R}_{s}(x)\cdot\nabla u^{R}_{s}(x)dx\bigg]ds+\int_{0}^{t}\sum_{k}\E[\norm{b'(u^{R}_{s})\phi_{k} \nabla u^{R}_{s}}_{L^{2}}^{2}]ds\nonumber\\
&\quad+\int_{0}^{t}\sum_{k}\E[\norm{b(u^{R}_{s}) \nabla\phi_{k}}_{L^{2}}^{2}]ds
\nonumber\\&\leqslant \norm{u^{R}_{0}}_{L^{2}}^{2}-\int_{0}^{t}\E[\norm{ \nabla u^{R}_{s}}_{L^{2}}^{2}]ds+C_{1}^W L^{2}\int_{0}^{t}\E[\norm{ \nabla u^{R}_{s}}_{L^{2}}^{2}]ds
+C_{2}^W \int_{0}^{t}\E[\norm{b(u^{R}_{s})}_{L^{2}}^{2}]ds.
\end{align}
Letting $\lambda:=1-C_{1}^W L^{2}>0$ and using the linear growth assumption on $b$ given by \eqref{eq:lin-g}, we obtain
\begin{align}\label{eq:u-b}
\E[\norm{u^{R}_{t}}_{L^{2}}^{2}]&\leqslant \norm{u^{R}_{0}}_{L^{2}}^{2}-\lambda\int_{0}^{t}\E[\norm{u^{R}_{s}}_{\dot{H}^{1}}^{2}]ds
+C_{2}^W C_{1}^b\int_{0}^{t}\E[\norm{u^{R}_{s}}_{L^{2}}^{2}]ds+C^W_{2}C^{b}_2 t.
\end{align}
Using Gronwall's inequality, we thus obtain
\begin{align}\label{eq:Gronwall-b}
\E[\norm{u^{R}_{t}}_{L^{2}}^{2}]\leqslant C_{2}^W C^{b}_1 \norm{u^{R}_{0}}_{L^{2}}^{2}\exp(t\tilde{\lambda})
\end{align} for $\tilde{\lambda}=C_{2}^W C^{b}_1$ and hence, plugging \eqref{eq:Gronwall-b} in \eqref{eq:u-b}, yields
\begin{align*}
\E[\norm{u^{R}_{t}}_{L^{2}}^{2}]+\lambda\int_{0}^{t}\E[\norm{u^{R}_{s}}_{H^{1}}^{2}]ds&\leqslant C_{2}^W C^{b}_2\norm{u^{R}_{0}}_{L^{2}}^{2}\exp(t\tilde{\lambda}),
\end{align*} which implies \eqref{eq:b1}, as 
$\norm{u_{0}^{R}}_{L^{2}}^{2}=\norm{\Pi_{R}u_{0}}_{L^{2}}^{2}\leqslant \norm{u_{0}}_{L^{2}}^2$.
\end{proof}

\begin{remark}[Energy estimate]\label{rem:energy-est}
If we take $b=\frac{1}{\sqrt{N}}f$, where  $f$ is given by \eqref{def:f}, we can improve the energy estimate by using that $\abs{f(x)}\leqslant\sqrt{\abs{x}}$ in order to estimate the $L^{2}$-norm of $b(u^{R}_{s})$ in \eqref{eq:b-est}. Utilizing also the mass conservation of the solution $u$, that we later prove in \cref{prop:mass conservation}, we then obtain the following a priori energy bound
\begin{align}\label{eq:energy-bound}
\E[\norm{u_{t}}_{L^{2}}^{2}]+\lambda\int_{0}^{t}\E[\norm{\nabla u_{s}}_{L^{2}}^{2}]ds&\leqslant \norm{u_{0}}_{L^{2}}^{2}+\frac{C_{2}^W}{N}\int_{0}^{t}\E[\norm{u_{s}}_{L^{1}}]ds\nonumber\\&\leqslant \norm{u_{0}}_{L^{2}}^{2}+\frac{C_{2}^W}{N}t\norm{u_{0}}_{L^{1}}.
\end{align}
\end{remark}

\begin{theorem}\label{thm:ex-sol}
Let Assumptions \ref{ass:noise}, \ref{ass:diff} and \ref{ass:const} hold and let $u_{0}\in L^{2}(\mathbb{T}^d)$. 
Then there exists a unique solution $u$ with paths in $L^{2}([0,T],H^{1}(\mathbb{T}^d))\cap C([0,T],L^{2}(\mathbb{T}^d))$ of the equation \eqref{eq:rSPDE} in the sense of \cref{def:app-sol}.
\end{theorem}

\begin{proof}
From \cref{axillary} it follows that for $v_{0}:=(1-\Delta)^{-1/2}u_{0}\in H^{1}\subset L^{2}$, there exists a unique strong solution $v\in L^{2}([0,T],H^{1})\cap C([0,T],L^{2})$ of \eqref{eq:v-VT}. Let $u_{t}:=(1-\Delta)^{1/2}v_{t}$. By the regularity of $v$, we obtain that almost surely
 \begin{equation*}
u \in  L^{2}([0,T],L^{2})\cap C([0,T],H^{-1}).
 \end{equation*}
Furthermore, from the equation of $v$, testing against $\varphi\in C^{\infty}(\mathbb{T}^{d})$, we obtain that $u$ solves the ``very weak'' equation \eqref{eq:H(-1)-form}. By \cref{lem:energy-est}, the Galerkin projected solutions $(u^{R})_{R}$ satisfy the energy bound \eqref{eq:b1}. 

\noindent Since $L^{2}(\Omega\times [0,T],H^{1})$ is reflexive, we thus obtain that, along a subsequence, $(u^{R})_{R}$ converges weakly in $L^{2}(\Omega\times [0,T],H^{1})$. As $v^{R}\to v$ in $L^{2}(\Omega\times [0,T],H^{1})$ (see above), 
it follows that for $R\to\infty$,
\begin{align*}
u^{R}=\Pi_{R}(1-\Delta)^{1/2}v^{R}\to (1-\Delta)^{1/2}v=u
\end{align*} in $L^{2}(\Omega\times [0,T],L^{2})$. Thus, the limit of each such subsequence is given by $u$ and we can conclude that the whole sequence $(u^{R})_{R}$ converges to $u$, weakly in $L^{2}(\Omega\times[0,T],H^{1})$.  In particular, the limit $u$ satisfies
  \begin{equation*}
  u\in L^{2}([0,T],H^{1})
  \end{equation*} 
almost surely. 
Due to $u\in L^{2}([0,T],H^{1})\cap C([0,T],H^{-1})$ a.s., the mapping $t\mapsto u_{t}\in L^{2}$ is almost surely weakly continuous. 
Since $u\in L^{2}([0,T],H^{1})$ a.s. and $u_{0}\in L^{2}$, and because \eqref{eq:H(-1)-form} is equivalent to $u$ solving
\begin{align*}
du_{t}=-\frac{1}{2}\Delta u_{t}dt+\sum_{k} \nabla [b(u_{t})\phi_{k}]\cdot dB_{t}^{k}\in (H^{1})^{\ast}=H^{-1},
\end{align*}
we can apply \cite[Theorem 4.2.5]{Liu2015} to obtain an Itô formula for $d\norm{u_{t}}_{L^{2}}^{2}$. Almost sure continuity of the integrals in time then implies almost sure continuity of the mapping 
\begin{equation}\label{norm_cont}
t\mapsto\norm{u_{t}}_{L^{2}}^{2}.
\end{equation}
From continuity of \eqref{norm_cont} and continuity of $t\mapsto u_{t}\in L^{2}$ in the weak topology we get that $u\in C([0,T],L^{2})$ almost surely. 
Hence, overall, we indeed have that $u\in L^{2}([0,T], H^{1})\cap C([0,T],L^{2})$ almost surely. By the regularity of $u$ and as $u$ solves \eqref{eq:H(-1)-form}, it follows that $u$ solves \eqref{eq:app-sol} (for all $\varphi\in C^{\infty}(\mathbb{T}^{d})$ and thus, by density for all $\varphi\in H^{1}(\mathbb{T}^{d})$). 
Uniqueness of the solution follows from \cref{equivalence}.
\end{proof}

\begin{remark}\label{rem:uR-conv}
Using the Itô formula for $\norm{u^{R}_{t}-u_{t}}_{L^{2}}^{2}$ (cf. \cite[Theorem 4.2.5]{Liu2015}) and as $u^{R}_{0}=\Pi_{R}u_{0}\to u_{0}$ in $L^{2}$, we see that $u^{R}\to u$ strongly in $L^{2}(\Omega\times [0,T], H^{1})$ and that the energy estimate holds true for the limit $u$.
\end{remark}

\begin{remark}\label{rem:cont}
The regularity, $u \in C([0,T],L^{2})$ will be used to prove the comparison principle, \cref{comparison}, below.
\end{remark}

\end{subsection}
\begin{subsection}{A comparison principle for regularized DK-type SPDEs}\label{sec:dk-com}
In this section we prove a comparison principle for the class of SPDEs \eqref{eq:rSPDE}, which will in particular imply positivity and mass conservation of the solution.

\begin{theorem}\label{comparison}
Let Assumptions \ref{ass:noise}, \ref{ass:diff} and \ref{ass:const} hold. Furthermore, let $u^{+}$ and $u^{-}$ be two solutions of \eqref{eq:app-sol} with initial conditions $u^{+}_{0},u^{-}_{0}\in L^{2}$, respectively, such that  $u^{+}_{0}(x)\geqslant u^{-}_{0}(x)$ for Lebesgue-almost all $x\in\mathbb{T}^{d}$. Then
\begin{align*}
\p \paren[\big]{u^{+}_{t}\geqslant u^{-}_{t}\,\operatorname{Leb}\text{-a.e.}\,\forall t\in [0,T]}=1,
\end{align*}
where $\operatorname{Leb}$ is the Lebesgue measure on $\mathbb T^d$.
\end{theorem}

\begin{proof}
We follow the proof of \cite[Theorem 2.1]{Donati-Martin93}. 
The main idea is an application of Itô's formula to a suitable $C^{2}$ approximation of the map $x\mapsto \max(x,0)^2$, applied to the difference of the solutions. More precisely, let for $p>0$, 
$\varphi_{p}\in C^2(\R,\R)$ be defined by 
\begin{equation*}
\varphi_{p}(x):=\mathbf{1}_{[0,\infty)}(x)\int_{0}^{x}\int_{0}^{y}[2pz\mathbf{1}_{[0,\frac{1}{p}]}(z)+2\,\mathbf{1}_{(\frac{1}{p},\infty)}(z)]dzdy.
\end{equation*}
Note that $\varphi_{p}$ satisfies
\begin{align*}
0\leqslant \varphi'_{p}(x)\leqslant 2 \max(x,0)\quad\text{ and }\quad 0\leqslant \varphi''_{p}(x)\leqslant 2\, \mathbf{1}_{x\geqslant 0}.
\end{align*} 
Next, we define 
\begin{align*}
\Phi_{p}(h):=\int_{\mathbb{T}^{d}}\varphi_{p}(h(x))dx.
\end{align*}  
Let $w_t:=u^{-}_t-u^{+}_t$ for $t>0$. Since  $\varphi_{p}(x)\uparrow\max(x,0)^{2}$ for $p\to\infty$, by monotone convergence we conclude that $\Phi_{p}(w_{t})\uparrow\norm{\max(w_{t},0)}_{L^{2}}^{2}$ for $p\to\infty$.
Moreover, $\Phi_{p}$ is twice Fréchet differentiable and we obtain by the Itô formula from \cite[Theorem 1.2]{Pardoux1979} that
\begin{align}
d\Phi_{p}(w_{t}) 
&=-\frac{1}{2}\langle \varphi''_p(w_{t}),\abs{ \nabla w_{t}}^2\rangle dt
+\sum_{k}\big\langle\varphi'_p(w_{t}), \nabla\paren[\big]{(b(u^{-}_{t})-b(u^{+}_{t}))\phi_{k}}\big\rangle \cdot dB_{t}^{k} \label{comp_mart} \\
&\quad+\frac{1}{2}\sum_{k}\int \abs[\big]{\nabla\paren[\big]{(b(u^{-}_{t})-b(u^{+}_{t}))\phi_{k}}}^{2}\varphi''_p(w_{t})dxdt. \label{comp}
\end{align}
Notice that $\Phi_{p}(w_{0})=0$, since by assumption $w_0 \leq 0$ a.e., and that the martingale term in \eqref{comp_mart} is a indeed a martingale.
From Young's weighted inequality for products we obtain $(x+y)^2 = x^2 + 2xy + y^2 \leqslant (1+\kappa)x^2 + (1+1/(2\kappa)) y^2$ for any $\kappa>0$, and using Assumption \ref{ass:diff} on $b$ we can thus estimate the quadratic variation term \eqref{comp} by
\begin{align*}
&\frac12 \sum_k \left( \nabla\paren[\big]{(b(u^{-}_{t})-b(u^{+}_{t}))\phi_{k}}\right)^{2}\varphi''_p(w_{t}) \\
&\leqslant \frac12 \sum_k \left[(1+\kappa)\left( \nabla(b(u^{-}_{t})-b(u^{+}_{t}))\phi_{k}\right)^{2} + (1+\frac{1}{2\kappa}) \left((b(u^{-}_{t})-b(u^{+}_{t})) \nabla\phi_{k} \right)^2 \right] \varphi''_p(w_{t})\\
&\leqslant \frac12 \left[(1+\kappa) C^W_1 L^2 |\nabla w_t|^2 + (1+\frac{1}{2\kappa})C^W_2 L^2 2\cdot \abs{w_{t}}^{2} \right] \varphi''_p(w_{t}).
\end{align*} 
Note that $\abs{w_{t}}^{2}\varphi''_p(w_{t}) \leqslant 2\max{(w_t,0)}^2$, and that $C_{1}^W L^2 < 1$ by Assumption~\ref{ass:const} which means that there exists $\kappa>0$ with $(1+\kappa)C_1^W L^2 < 1$. For such $\kappa$ we obtain 
\begin{align*}
\E[\phi_{p}(w_{t})]&\leqslant \frac12 \paren[\Big]{(1+\kappa) C_{1}^W L^2 -1}\int_{0}^{t}\!\E \left[ \langle \varphi''_p(w_{s}),\abs{ \nabla w_{s}}^2\rangle \right] ds \\
& \quad + 2 (1+\frac{1}{2\kappa})C^W_2 L^2\int_{0}^{t}\!\E \left[\norm{\max(w_{s},0)}_{L^{2}}^2 \right]ds\displaybreak[3]\\
&\lesssim \int_{0}^{t}\E \left[\norm{\max(w_{s},0)}_{L^{2}}^2 \right]ds.
\end{align*}
Taking $p\to \infty$ and applying  Gronwall's inequality yields
\begin{align*}
\E \left[\norm{\max(w_{s},0)}_{L^{2}}^2 \right]=0.
\end{align*}
Hence, $(\p\otimes\operatorname{Leb})(w_{t}\leqslant 0)=1$ for all $t\geqslant 0$. By continuity, $w\in C([0,T],L^{2})$, the claim follows.
\end{proof}

\begin{corollary}[Non-negativity of the solution]\label{cor:positivity}
Let Assumptions \ref{ass:noise}, \ref{ass:diff} and \ref{ass:const} hold, and assume additionally that $b(0)=0$. Let $u$ be a solution of \eqref{eq:rSPDE} with initial condition $u_{0} \geqslant 0$ almost everywhere. 
Then $\p(u_{t}\geqslant 0\,\operatorname{Leb}\text{-a.e.}\,\forall t\in [0,T])=1$. 
\end{corollary} 
\begin{proof}
Since $b(0)=0$, it follows that the zero function is a solution of \eqref{eq:rSPDE}. Then the claim directly follows from \cref{comparison}.
\end{proof}

\begin{proposition}[Conservation of mass]\label{prop:mass conservation}
Under the assumptions of the previous corollary, let $u$ solve \eqref{eq:rSPDE} with non-negative initial condition $u_{0}$. Then almost surely $\int\abs{u_{t}}(x)dx=\int \abs{u_{0}}(x)dx$ for all $t\in [0,T]$.
\end{proposition}
\begin{proof}
Using non-negativity of the solution $u$ obtained by \cref{cor:positivity} and  testing the equation against $\varphi = 1\in C^{\infty}(\mathbb{T}^{d})$, 
 we have that, for almost all $\omega\in\Omega$, 
\begin{align*}
\int\,\abs{u_{t}}(x)dx =\int u_{t}(x)dx = \langle u_{t},1\rangle=\langle u_{0},1\rangle =\int\,\abs{u_{0}}(x)dx .
\end{align*}
The claim follows from the continuity of $t \mapsto \int u_{t}(x)dx$.
\end{proof}

\end{subsection}
\end{section}
\begin{section}{Weak error estimate}\label{sec:error-est}
In this section we estimate the weak error between the  martingale solution $\mu^{N}$ of the Dean-Kawasaki equation and the strong solution $\tilde{\mu}^{N}$ of the approximate Dean-Kawasaki equation. For that purpose we first have to discuss the solution theory of the approximate Dean-Kawasaki equation:

\begin{proposition}\label{prop:app_mu_wellposed}
	Consider the equation
	\begin{equation}\label{def:app_mu_sec4}
	\begin{aligned}
		d\tilde{\mu}^{N}_{t} &=\frac{1}{2}\Delta\tilde{\mu}^{N}_{t}dt+\frac{1}{\sqrt{N}}\nabla\cdot (f(\tilde{\mu}^{N}_{t})dW^{N}_{t}), \\
		\tilde{\mu}_{0}^{N} & \in L^{2}(\mathbb{T}^d),
	\end{aligned}
	\end{equation}
	where $f \in C^1(\R)$ satisfies
	\begin{align}\label{eq:f-prime-bounds_sec4}
		\norm{f'}_{L^{\infty}}\lesssim\frac{1}{\sqrt{\delta}}, \quad 
		\abs{f'(x)}\lesssim\frac{1}{\sqrt{x}},\text{ for all } x>0
	\end{align}
	and
	\begin{align}\label{eq:f-bounds_sec4}
		\abs{f(x)}\lesssim\sqrt{\abs{x}}, \quad 
		\abs{f(x)^{2}-x}\lesssim\delta, \text{ for all } x\geqslant 0,
	\end{align} 
	and where
	\begin{equation}\label{app_noise-2}
		W^{N}_{t}(x):=\sum_{\abs{k}\leqslant M_{N}}e_{k}(x)B_{t}^{k}:=\sum_{\abs{k}\leqslant M_{N}}\exp(2\pi i k\cdot x)B_{t}^{k}
	\end{equation}
	for $d$-dimensional complex-valued Brownian motions $(B^{k})_{k\in\mathbb{Z}^{d}}$. If $\frac{C_N(2M_{N}+1)^{d}}{N\delta_N} < 1$ holds (``coercivity condition''), where $C_N>0$ is such that $\|f'\|_\infty^2 \leqslant \frac{C_N}{\delta_N}$, then Equation \eqref{def:app_mu_sec4} has a unique strong solution in the sense of Section~\ref{sec:dk-well}. If $\tilde \mu^N_0 \geqslant 0$, then $\tilde \mu^N_t \geqslant 0$ and $\norm{\tilde \mu^N_t}_{L^1}=\norm{\tilde \mu^N_0}_{L^1}$ for all $t \in [0,T]$.
\end{proposition}

\begin{proof}
	This follows from the results in Section~\ref{sec:dk-well} once we verify that our equation satisfies the assumptions of that section. For that purpose, note that
	\[
		C^W_1 = \sum_{\abs{k}\leqslant M_N} \norm{e_k}_\infty^2 = (2M_N+1)^d,\qquad C^W_2 = \sum_{\abs{k}\leqslant M_N} |2\pi k|^2 \norm{e_k}_\infty^2 <\infty,
	\]
	and, since $f(0)=0$, we also have $L^2 = C^b_1 \leqslant \frac{C_N}{N \delta_N}$ by assumption, and $C^b_2 = 0$. Therefore, Assumptions \ref{ass:noise} and \ref{ass:diff} hold. Moreover, the condition $\frac{C_N(2M_{N}+1)^{d}}{N\delta_N} < 1$ is  nothing else than $C^1_W \max\{L^2, C^b_1\} < 1$, i.e. also Assumption \ref{ass:const} holds.
\end{proof}

\noindent This well-posedness result together with the energy estimate of Section~\ref{sec:dk-well} would be sufficient to derive a weak error estimate for the approximation of $\mu^N$ with $\tilde \mu^N$. However, through the energy estimate the weak error would depend on $\norm{\tilde \mu^N_0}_{L^2}^2$. Therefore, inspired by \cite{fehrman2021well}, we derive an entropy estimate below. 

\noindent The main advantage is that this estimate only depends on $\int\tilde{\mu}_{0}^{N}\log(\tilde{\mu}_{0}^{N})$, which typically is much smaller than $\|\tilde\mu_0^N\|_{L^2}^2$ and this leads to better error estimates. To be precise, we will show that $\int\tilde{\mu}_{0}^{N}\log(\tilde{\mu}_{0}^{N})\lesssim\log(N)$, while typically $\norm{\tilde{\mu}_{0}^{N}}_{L^{2}}^{2}\lesssim N^{d}$, cf.~\cref{lem:initial-condition-entropy} below.
\begin{proposition}[Entropy estimate]\label{prop:entropy-est}
Let $\tilde{\mu}^{N}$ be a solution of the approximate Dean-Kawasaki equation \eqref{eq:app-sol} with the initial condition $\tilde{\mu}_{0}^{N}:=\rho^{N}\ast\mu_{0}^{N}$ for $\mu_{0}^{N}:=\frac{1}{N}\sum_{i=1}^{N}\delta_{x^{i}}$ and let $(\rho^{N})_{N\geqslant 1}$ be a mollifying sequence such that $\tilde{\mu}_{0}^{N}\geqslant 0$  
and $\norm{\rho^{N}\ast\mu_{0}^{N}}_{L^{1}}=1$. Furthermore, assume the  coercivity condition $\frac{C_N(2M_{N}+1)^{d}}{N\delta_N} < 1$, where $C_N>0$ is such that $\|f'\|_\infty^2 \leqslant \frac{C_N}{\delta_N}$. Then the following entropy estimate holds
\begin{align}\label{eq:entropy-est}
\MoveEqLeft
\sup_{t\in[0,T]}\E\bigg[\int\tilde{\mu}^{N}_{t}\log(\tilde{\mu}^{N}_{t})\bigg]+\lambda\int_{0}^{T}\E\bigg[\int\frac{\abs{ \nabla\tilde{\mu}^{N}_{t}}^{2}}{\tilde{\mu}^{N}_{t}}\bigg]dt\nonumber\\&\lesssim \int\tilde{\mu}^{N}_{0}\log(\tilde{\mu}^{N}_{0}) + \frac{T M_{N}^{d+2}}{N},
\end{align}
for $\lambda:=\frac14(1-\frac{C(2M_{N}+1)^{d}}{N\delta_N})$.
\end{proposition}

\begin{proof}
Let $\gamma > 0$ and $g_{\gamma}(y):=(\gamma+y)\log(\gamma+y), y\in [0,\infty)$. Then $g_\gamma\in C^{\infty}([0,\infty),\R)$, with $(g_{\gamma})'(y)=\log(\gamma+y)+1$ and $(g_{\gamma})''(y)=\frac{1}{\gamma+y}$, and we extend $g_\gamma$ to $\R$ so that the extension is in $C^2$. But recall that $\tilde \mu^N$ is positive by the comparison principle, and therefore we only evaluate $g_\gamma$ on $[0,\infty)$ below. By the Itô formula from \cite[Theorem 1.2]{Pardoux1979} we have
\begin{align*}
d\paren[\bigg]{\int g_\gamma(\tilde \mu^N_t)} & =-\frac12 \langle g_\gamma''(\tilde \mu^N_t), |\nabla \tilde \mu^N_t|^2 \rangle dt + \sum_{|k|\leqslant M_N}\big\langle g_\gamma'(\tilde \mu^N_t), \nabla\paren{f(\tilde \mu^N_t)e_k}\big\rangle \cdot dB_{t}^{k} \\
&\quad+\frac{1}{2N}\sum_{|k|\leqslant M_N} \langle g_\gamma''(\tilde \mu^N_t),\abs{\nabla\paren{f(\tilde \mu^N_t)e_k}}^{2}\rangle dt\\
& =  \paren[\bigg]{- \frac12\int\frac{|\nabla \tilde \mu^N_t|^2}{\gamma+\tilde \mu^N_t} + \frac{1}{2N} \sum_{|k|\leqslant M_N} \int \frac{\abs{\nabla\paren{f(\tilde \mu^N_t)e_k}}^{2}}{\gamma+\tilde \mu^N_t}}dt + dM_t,
\end{align*}
where $M$ denotes the local martingale term. We estimate the It\^o correction term for $\kappa>0$ (applying Young's inequality in the same way as in the proof of the comparison principle) by
\begin{align*}
\int\frac{\abs{ \nabla(f(\tilde \mu^N_t)e_{k})}^{2}}{\gamma + \tilde \mu^N_t}
	& \leqslant (1+\kappa)\int\frac{\abs{ \nabla(f(\tilde \mu^N_t))e_{k}}^{2}}{\gamma + \tilde \mu^N_t} + (1+\frac{1}{2\kappa})\frac{\abs{ f(\tilde \mu^N_t)\nabla e_{k}}^{2}}{\gamma + \tilde \mu^N_t} \\
	& \leqslant (1+\kappa) \|f'\|_\infty^2 \int\frac{|\nabla \tilde \mu^N_t|^2}{\gamma+\tilde \mu^N_t} + (1+\frac{1}{2\kappa}) \norm[\bigg]{\frac{f}{\sqrt{\cdot}}}_\infty |2\pi k|^2.
\end{align*}
Summing over $k$ and estimating $\|f'\|_\infty^2 \leqslant C/\delta_N$, we obtain for some $K > 0$
\[
	d\paren[\bigg]{\int g_\gamma(\tilde \mu^N_t)} \leqslant \paren[\bigg]{\frac12 \paren[\bigg]{(1+\kappa)\frac{C(2M_N+1)^d}{N\delta_N}-1} \int\frac{|\nabla \tilde \mu^N_t|^2}{\gamma+\tilde \mu^N_t} + (1+\frac{1}{2\kappa})K\frac{M_N^{d+2}}{N} }dt + dM_t,
\]
and by choosing $\kappa>0$ appropriately we can achieve that $\frac12 \paren{(1+\kappa)\frac{C(2M_N+1)^d}{N\delta_N}-1} = -\lambda$. Next, we show that the local martingale is a true martingale: Since $\abs{f(\mu)}^{2}\lesssim \abs{\mu}$, its quadratic variation satisfies
\begin{align*}
\sum_{\abs{k}\leqslant M_{N}}\int_{0}^{t}\E\bigl[\abs[\big]{\bigl\langle\log(\gamma + \tilde \mu^N_s)+1,\nabla(f(\tilde{\mu}^{N}_{s})e_{k})\bigr\rangle}^{2}\bigr]ds & = \sum_{\abs{k}\leqslant M_{N}}\int_{0}^{t}\E\bigg[\abs[\bigg]{\int\frac{(f(\tilde \mu^N_s)e_{k} )\nabla u_{s}}{\gamma + \tilde \mu^N_s}}^{2}\bigg]ds \\
	& \lesssim \sum_{\abs{k}\leqslant M_{N}}\int_{0}^{t}\E[\abs{\nabla \tilde \mu^N_s}^{2}]ds < \infty.
\end{align*}
Together, the above yields the bound
\begin{align*}
\MoveEqLeft
\E\bigg[\int(\gamma+\tilde \mu^N_t)\log(\gamma+\tilde \mu^N_t)\bigg]+\lambda\int_{0}^{t}\E\bigg[\int\frac{\abs{ \nabla \tilde \mu^N_s}^{2}}{\gamma+\tilde \mu^N_s}\bigg]ds\\
	&\lesssim \int(\gamma+\tilde \mu^N_0)\log(\gamma+\tilde \mu^N_0) + \frac{t M_{N}^{d+2}}{N}
\end{align*}
As $[0,\infty)\ni x\to x\log(x)$ is bounded from below, we can apply Fatou's lemma to the left-hand side, and since $\abs{x\log(x)}\leqslant x^{2}+1$ for $x\geqslant 0 $ and $\norm{\tilde{\mu}^{N}_{0}}_{L^{2}}^{2}<\infty$, we can apply the dominated convergence theorem to the right-hand side, such that we obtain for $\gamma\to 0$,
\begin{align*}
\E\bigg[\int \tilde{\mu}^{N}_{t}\log(\tilde{\mu}^{N}_{t})\bigg]+\lambda\int_{0}^{t}\E\bigg[\int\frac{\abs{ \nabla\tilde{\mu}^{N}_{s}}^{2}}{\tilde{\mu}^{N}_{s}}\bigg]ds\lesssim \int \tilde{\mu}^{N}_{0}\log(\tilde{\mu}^{N}_{0}) + \frac{t M_{N}^{d+2}}{N}.
\end{align*}
The claim now follows by taking the supremum in $t \in [0,T]$. 
\end{proof}

\noindent As discussed above, there are good approximations of Dirac masses with moderately large entropy:
\begin{lemma}\label{lem:initial-condition-entropy}
	Let $\rho \in C^\infty_c(\R^d)$ with $\int_{\mathbb{R}^{d}}\rho(y)dy=1$, such that $\rho\geqslant 0$ and $\rho$ is symmetric in the sense that $\rho(x)=\rho(-x)$ for all $x\in\mathbb{R}^{d}$. Define $\rho^N(y) := \sum_{k \in \mathbb Z^d} N^d \rho(N(y+k))$, $y \in \mathbb T^d$. Then $\rho^N \in C^\infty(\mathbb T^d)$ is positive, $\int_{\mathbb T^d} \rho^N(y)dy=1$, and with $\mu^N_0 = \sum_{i=1}^N \delta_{x_i}$ and $\tilde \mu^N_0 := \mu^N \ast \rho^N$ we have
	\[
		\tilde \mu^N_0 \geqslant0,\qquad \|\tilde \mu^N_0 \|_{L^1(\mathbb T^d)} = 1, \qquad \int\tilde\mu_{0}^{N}\log(\tilde\mu_{0}^{N}) \lesssim \log N.
	\]
	Moreover, we have for all $\varphi \in C^2_b(\R^d)$:
	\[
		\abs{\langle \mu^N_0, \varphi\rangle - \langle \tilde \mu^N_0, \varphi\rangle} \lesssim N^{-2}\norm{\varphi}_{C^2_b},
	\]
	where $\norm{\varphi}_{C^2_b}:= \sum_{j=0}^2 \norm{D^j \varphi}_\infty$.
\end{lemma}

\begin{proof}
	Positivity is obvious because $\tilde \mu^N_0$ is a convolution of positive measures. The statement about the $L^1$ norm follows by Fubini's theorem. For the entropy, we first trivially bound
\begin{align*}
	\norm{\tilde\mu_{0}^{N}}_{L^{\infty}(\mathbb T^d)} = \norm{\rho^{N}\ast\mu_{0}^{N}}_{L^{\infty}(\mathbb T^d)}\leqslant N^{d}\norm{\rho}_{L^{\infty}}\lesssim N^{d}.
\end{align*}
Since $x \log x$ is negative on $[0,1)$, we get together with the trivial $L^\infty$ bound:
\begin{align*}
\int(\tilde\mu_{0}^{N})\log(\tilde\mu_{0}^{N}) &\leqslant\int(\tilde\mu_{0}^{N})\log(\tilde\mu_{0}^{N})\mathbf{1}_{\tilde\mu_{0}^{N}\geqslant 1}
\\&\leqslant \norm{\tilde\mu_{0}^{N}}_{L^{1}}\norm{\log(\tilde\mu_{0}^{N})\mathbf{1}_{\tilde\mu_{0}^{N}\geqslant 1}}_{L^{\infty}}
\\&\leqslant \log(\norm{\tilde\mu_{0}^{N}}_{L^{\infty}})
\\&\lesssim \log(N^{d})
=d\log(N).
\end{align*} 
To compare $\langle \mu^N_0, \varphi\rangle$ and $\langle \tilde \mu^N_0, \varphi\rangle$, note that
\begin{align*}
	\abs{\langle \mu^N_0, \varphi\rangle - \langle \tilde \mu^N_0, \varphi\rangle} & = \abs[\bigg]{\frac1N \sum_{i=1}^N \varphi(x_i) - \frac1N \sum_{i=1}^N \sum_{k \in \mathbb Z^d} \int_{\mathbb T^d} N^d \rho(N(x_i+k-y)) \varphi(y)dy} \\
	& = \abs[\bigg]{\frac1N \sum_{i=1}^N (\varphi(x_i) - \psi^n\ast \varphi(x_i))} \leqslant \norm{\varphi - \psi^N\ast \varphi}_\infty,
\end{align*}
where $\psi^N(x) = N^d \rho(Nx)$ for $x \in \R^d$, and where the convolutions in the second line are on $\R^d$ and not on $\mathbb T^d$. Using the symmetry of $\rho$, we get for all $x \in \R^d$
\begin{align*}
\abs{(\varphi - \psi^N\ast \varphi)(x)}&=\abs[\bigg]{\int_{\mathbb{R}^{d}}\rho(Ny)N^{d}\paren[\Big]{\varphi(x)-\nabla \varphi(x) \cdot y -\varphi(y)}dy}
\\&\leqslant \norm{\varphi}_{C^2_b} \abs[\bigg]{\int_{\mathbb{R}^{d}}\rho(Ny)N^{d}|y|^2 dy} = N^{-2}\|\varphi\|_{C^2_b} \int_{\R^d} \rho(y)|y|^2 dy,
\end{align*}
which concludes the proof.
\end{proof}

\noindent Note that the entropy of $\tilde \mu^N_0$ is much smaller than its $L^{2}$ norm, which in general we can only bound by $N^{d}$.

\begin{theorem}\label{thm:error-est}
Let $\mu^{N}$ be the martingale solution of the Dean-Kawasaki equation in the sense of \cref{def:DK} with initial condition $\mu_{0}^{N}:=\frac{1}{N}\sum_{i=1}^{N}\delta_{x^{i}}$. Let $\tilde \mu^N_0$ be as in Lemma~\ref{lem:initial-condition-entropy} and let $f$ and $W^N$ be as in \cref{prop:app_mu_wellposed}. Assume that with the notation of \cref{prop:app_mu_wellposed}
 $\sup_N \frac{C_N(2M_{N}+1)^{d}}{N\delta_N} < 1$. Let $\tilde{\mu}^{N}$ be the solution of the approximate Dean-Kawasaki equation \eqref{def:app_mu_sec4} with initial condition $\tilde{\mu}_{0}^{N}$.\\
Then for any $t>0$, $\varphi\in C^{\infty}(\mathbb{T}^{d})$ and $F(\mu):=\exp(\langle\mu,\varphi\rangle)$ for $\mu\in\mathcal{M}$, the following weak error bound holds:
\begin{align}
	\abs{\E[F(\tilde{\mu}^{N}_{t})]-\E[F(\mu^{N}_{t})]} \lesssim_{\varphi} N^{-2} +\frac{\delta_N}{N}t+\frac{t}{M_{N}^2N}+ \frac{M_{N}^{d}}{N^2}t+ \frac{\log N}{M_{N}^{2}N}.	
\end{align}
For $M_N = \delta_N^{-1/2}$ and $\delta_N \simeq N^{-\frac{1}{d/2+1}}$ (which is the optimal choice under the coercivity condition) we have
\begin{align}\label{eq:weak-error}
\abs{\E[F(\tilde{\mu}^{N}_{t})]-\E[F(\mu^{N}_{t})]}\lesssim_{\varphi} N^{-1-\frac{1}{d/2+1}}(t+\log(N)).
\end{align} 
\end{theorem}

\begin{remark}
With the functions $F_k(\mu)=\exp(\langle\mu,\varphi_k\rangle)$ for a suitable dense set $(\varphi_k)_{k \in \N}\subset C^{\infty}(\mathbb T^d)$, we could use \cite[Theorem 3.4.5]{Ethier1986} to construct a metric for the topology of weak convergence of probability measures on $\mathcal{M}$ and then replace the left-hand side of \eqref{eq:weak-error} by the distance of $\tilde{\mu}^{N}_{t}$ and $\mu^{N}_{t}$ in this metric.
\end{remark}

\begin{proof}[Proof of \cref{thm:error-est}]
To prove the weak error bound \eqref{eq:weak-error}, we apply the duality argument of~\cite{Konarovskyi2019}. 
For that purpose, let $v$ solve the Hamilton-Jacobi equation with initial condition $\varphi\in C^{\infty}(\mathbb{T}^{d})$,
\begin{align}\label{eq_for_v}
\partial_{t}v =\frac{1}{2}\Delta v +\frac{1}{2N}\abs{\nabla v}^{2},\quad v_{0} =\varphi. 
\end{align}
The Cole-Hopf transformation $w = e^{\frac1N v}$ solves $\partial_t w = \frac12 \Delta w$, and therefore with the heat kernel $p$ of $\frac12 \Delta$ (on $\mathbb T^d$ with periodic boundary conditions):
\[
	v_t = N\log (p_t \ast e^{\frac1N \varphi}). 
\]
By the maximum principle we have $v_t \in [\min_{x \in \mathbb T^d} \varphi(x), \max_{x \in \mathbb T^d} \varphi(x)]$ for all $t\geqslant 0$, and by the explicit representation of the solution we get
\[
	|\partial_i v_t| = \abs[\bigg]{N\frac{p_t \ast \paren[\big]{\frac1N \partial_i \varphi e^{\frac1N\varphi}}}{p_t \ast e^{\frac1N \varphi}}} \lesssim_\varphi 1
\]
uniformly in $N$ and $t$, and by a similar application of the chain rule also $|\partial_{ij} v_t|\lesssim_\varphi 1$ uniformly in $N$ and $t$, so overall $\sup_{t\geqslant 0, N\in \N} \norm{v}_{C^2_b} \lesssim_\varphi 1$. 

\noindent By definition, we have $F(\mu^{N}_{t})=\exp(\langle\mu^{N}_{t},v_{t-t}\rangle)$ and with the Laplace duality of \cite[Theorem 2.3]{Konarovskyi2019} we obtain
\[
	\E[F(\mu^{N}_{t})] = \exp(\langle \mu_0^N, v_t\rangle).
\]

For $\tilde{\mu}^{N}$ we have on $[0,t]$
\begin{align} \label{eq:tilde-mu-eq}
d\langle \tilde{\mu}_{s}^{N},v_{t-s}\rangle = \langle\tilde{\mu}_{s}^{N},\paren[\big]{-\partial_s + \frac{1}{2}\Delta}v_{t-s}\rangle ds - \frac{1}{\sqrt{N}}\sum_{\abs{k}\leqslant M_{N}} \langle f(\tilde{\mu}_{s}^{N})e_{k}, \nabla v_{t-s}\rangle dB^{k}_{s},
\end{align}
and using equation \eqref{eq_for_v} for $v$, we arrive at
\begin{align}\label{eq:mart2}
d\paren[\big]{\exp(\langle\tilde{\mu}^{N}_{s},v_{t-s}\rangle)}_{s} &=-\frac{1}{\sqrt{N}}\sum_{\abs{k}\leqslant M_{N}}\exp(\langle\tilde{\mu}^{N}_{s},v_{t-s}\rangle)\langle f(\tilde{\mu}_{s}^{N})e_{k},\nabla v_{t-s} \rangle dB_{s}^{k}\\
&\quad+\frac{1}{2N}\exp(\langle\tilde{\mu}^{N}_{s},v_{t-s}\rangle)\paren[\bigg]{\sum_{\abs{k}\leqslant M_{N}}\!\!\abs{\langle f(\tilde{\mu}^{N}_{s})e_{k},\nabla v_{t-s}\rangle}^{2}-\langle\tilde{\mu}^{N}_{s},\abs{\nabla v_{t-s}}^{2}\rangle}ds.\nonumber
\end{align}  
Note that the stochastic integral on the right hand side is indeed a martingale (not only a local martingale). This can be concluded from the bounds $\exp\paren{\langle \tilde\mu^N_s, v_{t-s}\rangle}\lesssim 1$ (by mass conservation of $\tilde \mu^N$ and the uniform bounds for $v$), $|\nabla v_{t-s}|\lesssim 1$ and $f(u)^{2}\lesssim u$, which yield
\begin{align*}
\sum_{\abs{k}\leqslant M_{N}}\int_{0}^{t}\E[\exp(\langle\tilde{\mu}^{N}_{s},v_{t-s}\rangle)\abs{\langle f(\tilde{\mu}^{N}_{s})e_{k}, \nabla v_{t-s}\rangle}^{2}]ds & \leqslant \int_{0}^{t}\E[\norm{f(\tilde{\mu}^{N}_{s}) \nabla v_{t-s}}_{L^2}^{2}]ds \\
	& \lesssim \int_{0}^{t}\E[\norm{\tilde{\mu}^{N}_{s}}_{L^{1}}]ds = t.
\end{align*} 
Writing~\eqref{eq:mart2} in integral form and taking the expectation, we thus obtain the bound
\begin{align}
&\abs{\E[F(\tilde{\mu}^{N}_{t})]-\E[F(\mu^{N}_{t})]} \nonumber\\
&\leqslant \abs{\exp(\langle\tilde \mu_{0}^{N},v_{t}\rangle) - \exp(\langle\mu_{0}^{N},v_{t}\rangle)} \nonumber\\
&\quad + \frac{1}{2N}\E\bigg[\int_{0}^{t}\exp(\langle\tilde{\mu}^{N}_{s},v_{t-s}\rangle)\abs[\bigg]{\sum_{\abs{k}\leqslant M_{N}} \!\!\! \abs{\langle f(\tilde{\mu}^{N}_{s})e_{k},\nabla v_{t-s}\rangle}^{2}-\langle\tilde{\mu}^{N}_{s},\abs{\nabla v_{t-s}}^{2}\rangle}ds\bigg]. \label{bracketTerm}
\end{align} 
Since $\exp(\langle\tilde{\mu}^{N}_{s},v_{t-s}\rangle) \lesssim 1$, we are left with estimating the term in the brackets in (\ref{bracketTerm}). 
To that aim, we write
\begin{align*}
\frac{1}{2N}\paren[\bigg]{\sum_{\abs{k}\leqslant M_{N}}\abs{\langle f(\tilde{\mu}^{N}_{s})e_{k},\nabla v_{t-s}\rangle}^{2}-\langle\tilde{\mu}^{N}_{s},\abs{\nabla v_{t-s}}^{2}\rangle} = \frac{A_{s}+C_{s}}{2N},
\end{align*}
where 
\begin{equation*}
A_{s} :=  \langle f(\tilde{\mu}_{s}^N)^2, |\nabla v_{t-s}|^2 \rangle -  \langle \tilde{\mu}_{s}^N , |\nabla v_{t-s}|^2 \rangle
\end{equation*}
and
\begin{align*}
C_{s}  :=&  \sum_{\abs{k}\leqslant M_{N}}\abs{\langle f(\tilde{\mu}^{N}_{s})e_{k},\nabla v_{t-s}\rangle}^{2} - \langle f(\tilde{\mu}_{s}^N)^2, |\nabla v_{t-s}|^2 \rangle \\
	 = &\sum_{\abs{k}\leqslant M_{N}}\abs{\langle f(\tilde{\mu}^{N}_{s})\nabla v_{t-s},e_{k}\rangle}^{2} - \norm{f(\tilde \mu^N_s) \nabla v_{t-s}}_{L^2}^2.
\end{align*}
To estimate $A$, we use positivity of $\tilde{\mu}^{N}$ from \cref{cor:positivity}, as well as \eqref{eq:f-bounds_sec4}, such that $\abs{x-f(x)^{2}}\lesssim\delta_N$ for $x\geqslant 0$:
\begin{align}\label{Aestimate}
\abs{A_{s}}=\abs{\langle\tilde{\mu}^{N}_{s},\abs{\nabla v_{t-s}}^{2}\rangle-\langle f(\tilde{\mu}^{N}_{s})^{2},\abs{\nabla v_{t-s}}^{2}\rangle}\lesssim\delta_N\int\abs{\nabla v_{t-s}}^{2}(x)dx \lesssim \delta_N.
\end{align}
In order to estimate the term $C_s$, we apply Parseval's identity and obtain with $g_s:=f(\tilde{\mu}^{N}_{s})\nabla v_{t-s}$
\begin{align*}
\abs{C_{s}}&=\sum_{\abs{k}> M_{N}}\abs{\langle g_s,e_{k}\rangle}^{2} \leqslant M_{N}^{-2}\sum_{\abs{m}>M_{N}}\abs{m}^{2}\abs{\hat{g}_s(m)}^{2} \leqslant M_{N}^{-2}\norm{\nabla g_s}_{L^2}^{2}.
\end{align*}
Now we get from the chain rule
\[
	\norm{\nabla g_s}_{L^2}^2\lesssim \norm{f(\tilde \mu^N_s)}_{L^2}^2 \norm{v_{t-s}}_{C^2_b}^2 + \norm{f'(\tilde \mu^N_s) \nabla \mu^N_s}_{L^2}^2 \norm{v_{t-s}}_{C^1_b}^2 \lesssim \norm{v_{t-s}}_{C^2_b} \paren[\bigg]{\norm{\tilde \mu^N_s}_{L^1} + \int \frac{\abs{\nabla \tilde \mu^N_s}^2}{\tilde \mu^N_s}},
\]
where we used that 
$f(x)^{2}\lesssim \abs{x}$ and $f'(x)^{2}\lesssim\frac{1}{x}$ for $x>0$ by \eqref{eq:f-prime-bounds_sec4}. Integrating over time, using the entropy estimates from \cref{prop:entropy-est} to bound the second term and conservation of mass, $\norm{\tilde{\mu}^{N}_{s}}_{L^{1}}=1$, for the first term (here we need the uniform-in-$N$ coercivity bound to not pick up some $\lambda_N^{-1}$ in the estimates), we obtain
\begin{align}\label{Cestimate}
\int_{0}^{t}\E[\abs{C_{s}}]ds\lesssim_{\varphi} \frac{1}{M_{N}^2}\paren[\bigg]{t+\frac{tM_{N}^{d+2}}{N}+\int\tilde{\mu}_{0}^{N} \log(\tilde{\mu}_{0}^{N})} \lesssim \frac{1}{M_{N}^2} \paren[\bigg]{t+\frac{tM_{N}^{d+2}}{N}+ \log N},
\end{align}
where the last bound holds by Lemma~\ref{lem:initial-condition-entropy}. We estimate the error from the initial condition with the mean value theorem as follows:
\begin{align}\label{ICestimate}
\abs{\exp(\langle\tilde{\mu}_{0}^{N},v_{t}\rangle) - \exp(\langle\mu_{0}^{N},v_{t}\rangle)} \leqslant \abs{\langle\tilde{\mu}_{0}^{N}-\mu_{0}^{N},v_{t}\rangle}\exp(\norm{v_{t}}_{\infty})\lesssim_\varphi N^{-2},
\end{align}
where the last bound holds again by Lemma~\ref{lem:initial-condition-entropy}. Plugging \eqref{Aestimate}--\eqref{ICestimate} into \eqref{bracketTerm}, we obtain for the weak error:
\begin{align*}
	\abs{\E[F(\tilde{\mu}^{N}_{t})]-\E[F(\mu^{N}_{t})]} \lesssim_{\varphi} N^{-2} +\frac{\delta_N}{N}t+\frac{t}{M_{N}^2N}+ \frac{M_{N}^{d}}{N^2}t+ \frac{\log N}{M_{N}^{2}N}.
\end{align*}
The coercivity assumption dictates $M_{N}^{d}\lesssim\delta_N N$. Hence, if we estimate $M_{N}^{d}$ by $\delta_N N$ in the fourth term, we get
\begin{align}
\abs{\E[F(\tilde{\mu}^{N}_{t})]-\E[F(\mu^{N}_{t})]} \lesssim_{\varphi} N^{-2}+\frac{\delta_N}{N}t+\frac{t}{M_{N}^2 N}+\frac{\log N}{M_{N}^2N}.
\end{align}
Choosing $M_{N}=\delta_N^{-1/2}$ and respecting the coercivity assumption, we arrive at the optimal choice $\delta_N \simeq N^{-\frac{1}{d/2+1}}$. Altogether, we obtain the estimate
\begin{align*}
\abs{\E[F(\tilde{\mu}^{N}_{t})]-\E[F(\mu^{N}_{t})]} \lesssim_{\varphi}N^{-2}+N^{-1-\frac{1}{d/2+1}}(t+\log N) \lesssim N^{-1-\frac{1}{d/2+1}}(t+\log N).
\end{align*}

\end{proof}
\end{section}
\begin{section}*{Acknowledgements}
We gratefully acknowledge funding by Deutsche Forschungsgemeinschaft (DFG) under Germany’s Excellence Strategy – The Berlin Mathematics Research
Center MATH+ (EXC-2046/1, project ID: 390685689) and through the grant CRC
1114 ``Scaling Cascades in Complex Systems”, Project Number 235221301, Project C10 ``Numerical analysis for nonlinear SPDE models of particle systems".\\
We thank Xiaohao Ji for carefully reading a preliminary version of the manuscript.
\end{section}
\bibliographystyle{alpha}
\bibliography{Literatur}
\end{document}